\documentclass[12pt]{article}

\usepackage{amsmath}
\usepackage{amsfonts}
\usepackage{amssymb}
\usepackage{amscd}
\usepackage{amsthm}
\usepackage{textcomp}
\usepackage{bbm}
\usepackage{color}
\usepackage[linktocpage=true,plainpages=false,pdfpagelabels=false]{hyperref}

\input xypic

\usepackage[matrix,arrow,curve]{xy}

\usepackage{color}

\newcommand{\quash}[1]{}


\newtheorem{defin}{Definition}
\newtheorem{prop}{Proposition}
\newtheorem{nt}{Remark}
\newtheorem{Th}{Theorem}
\newtheorem{lemma}{Lemma}
\newtheorem{cons}{Corollary}

\newtheorem{defin-prop}{Definition-proposition}

\newfont{\ssdbl}{msbm8}
\newfont{\sdbl}{msbm9}
\newfont{\dbl}{msbm10 at 12pt}

\newcommand{\oo}{{\cal O}}
\newcommand{\ff}{{\cal F}}

\newcommand{\g}{{\cal G}}
\newcommand{\ad}{{\cal A}}

\newcommand{\res}{\mathop {\rm res}}

\newcommand{\Dim}{\mathop {\rm Dim}}
\newcommand{\Hom}{\mathop {\rm Hom}}
\newcommand{\Aut}{\mathop {\rm Aut}}

\newcommand{\tr}{\mathop {\rm Tr}}

\newcommand{\Spec}{\mathop {\rm Spec}}

\newcommand{\Frac}{\mathop {\rm Frac}}

\newcommand{\da}{\mathbb{A}}
\newcommand{\dz}{\mathbb{Z}}

\newcommand{\Z}{\dz}

\newcommand{\Ker}{{\rm Ker}\:}

\newcommand{\f}{{\cal F}}
\newcommand{\lrto}{\longrightarrow}

\newcommand{\E}{{\cal E}}

\def\Z{{\mathbb Z}}

\newcommand{\diag}{{\rm diag}}

\begin{document}

\author{
D. V. Osipov
}

\title{Central extensions and Riemann-Roch theorem on algebraic surfaces \thanks{The  author was supported by the Basic Research Program of the National Research University Higher School of Economics.}}
\date{}

\maketitle

\begin{abstract}
We study   canonical central extensions of the general linear group over the ring of adeles on a smooth  projective algebraic surface $X$ by means of the group of integers.
By these central extensions and adelic transition matrices of a rank $n$ locally free sheaf of $\oo_X$-modules  we obtain the local (adelic) decomposition for the difference of Euler characteristics of this sheaf   and the sheaf $\oo_X^n$. Two various calculations of this difference lead to the Riemann-Roch theorem on $X$ (without the Noether formula).
\end{abstract}

\section{Introduction}  \label{Intro}

This paper is about locally free sheaves, adeles and the Riemann-Roch theorem on algebraic surfaces.

But first, we briefly recall the well-known case of algebraic curves.

Recall  that two vector subspaces $A$ and $B$
in a vector space $V$ over a field $k$ are commensurable (see~\cite{T}), i.e. $A \sim B$, iff
$$
\dim_k (A + B)/ (A \cap B)   < \infty  \, \mbox{.}
$$
For such $A$ and $B$ we denote their relative dimension
$$
[A \mid B] = \dim_k B / (A \cap B) - \dim_k A / (A \cap B)  \, \mbox{.}
$$

We fix a $k$-vector subspace $K$ in $V$ and any $k$-vector subspaces $D_1 $ and $D_2 $ in $V$
such that
$D_1 \sim D_2$ and
$$ H^0(D_i)= D_i \cap K        \qquad \qquad H^1(D_i)= V/ (D_i + K)$$
are finite-dimensional $k$-vector spaces for $i \in \{1,2\}$.
We denote for $i \in \{ 1,2 \}$
$$\chi (D_i)  = \dim_k H^0(D_i) - \dim_k H^1(D_i) \, \mbox{.} $$

Then we have (see, e.g., \cite[\S~14.14, Exer.~14.59--14.61]{K}) ``the abstract Riemann-Roch theorem'':
\begin{equation} \label{ARR}
 \chi(D_1) - \chi(D_2) = [D_2   \mid D_1] \, \mbox{.}
\end{equation}

Let $n \ge 1$ be an integer.

For a smooth projective curve $S$ over $k$ we consider
$$V = \da_S^n \, \mbox{,} \qquad  \mbox{where} \qquad \da_S = {\mathop{\prod}\limits_{p \in S}}' K_p$$
 is  the space of adeles of the curve $S$,
and $K_p$ is the field of fraction of the completion $\hat{\oo}_p$ of the local ring $\oo_p$ at a (closed) point $p \in S$. We consider $K = k(S)^n$,
 where $k(S)$ is the field of rational functions on~$S$.

Now we consider  a  locally free sheaf  $\ff$ of $\oo_S$-modules of rank $n$  on $S$.
The stalk of $\ff$ at the generic point $\Spec k(S)$ of $S$
 is a $k(S)$-vector space.
 We fix a basis $e_0$ of this vector space.   For any (closed) point  $p \in S$,
 the completion of the stalk $\ff_p$ of $\ff$ at $p$ is a free $\hat{\oo}_p$-module. We
   fix a basis $e_p$ of this module. For all  (closed) points  $p \in S$, we consider the transition matrices
$\gamma_{01,p} \in GL_n(K_p)$ defined by the equality $e_0 = \gamma_{01,p} e_p $, which is calculated in the  $K_p$-vector space $\ff_p \otimes_{\oo_p} K_p$. The element given by the collection of matrices
$$\gamma_{01, \ff}=  \prod\limits_{p \in S} \gamma_{01, p}  \;  \in  GL_n \left( \prod_{p \in S} K_p \right) \mbox{,}$$
 belongs to the subgroup
$GL_n(\da_S)$.

We note that the chosen and fixed basis $e_0$ gives an embedding of $\ff$ to a constant sheaf $K$ on $S$.
Therefore we can associate by $\ff$
a $k$-vector subspace $D_{\ff}$  in $V$ explicitly given~as
$$D_{\ff}= \gamma_{01, \ff} D \,  \mbox{,} \qquad   \mbox{where} \qquad
D =(\prod_{p \in S} \hat{\oo}_p)^n $$
is a $k$-vector subspace of $V$.

 We note that $D_{\ff} \sim D$.

Now, applying formula~\eqref{ARR} for the $k$-vector subspaces $D_{\ff}$ and $D$  and using the adelic complex for $\ff$ on $S$ that gives
$$
H^i(D_{\ff}) = H^i(S, \ff)  \qquad \mbox{and}  \qquad
H^i(D)= H^i(S, \oo_S^n)
$$
for $i \in \{0, 1 \}$, we obtain the Riemann-Roch theorem for $\ff$ on~$S$:
\begin{equation}  \label{diff-1}
\chi(\ff)  - n \chi(\oo_S) = c_1( \ff)  \, \mbox{,}
\end{equation}
where $\chi (\ff)$ is the Euler characteristic of the sheaf $\ff$  on $S$.
And the first Chern number $c_1(\ff)$
 is obtained from the homomorphism of groups:
\begin{equation}  \label{deg}
\deg \: : \; GL_n(\da_S)   \lrto \Z     \quad \mbox{,}   \quad  a \longmapsto [a D \mid D]
\end{equation}
such that $c_1(\ff) = \deg(\gamma_{01, \ff})$.

This can be called the {\em local (or adelic) decomposition} for the difference of Euler characteristics of the sheaf $\ff$ and the sheaf $\oo_S^n$.

\medskip

Now let $X$ be a smooth projective algebraic surface over $k$. Then there is the Parshin-Beilinson adelic ring $\da_X$ of $X$
(see~\cite{Par, B, H, Osip1, Par6} and Section~\ref{sec-centr} below). We have
$$
\da_X = {\prod_{x \in C}}' K_{x,C}  \subset \prod_{x \in C} K_{x,C}  \, \mbox{,}
$$
where the (``two-dimensional'') adelic product is taken over all pairs $x \in C$, where $C$ is an irreducible curve on $X$, and $x$ is a point on $C$, and an Artinian ring $K_{x,C}$ is a finite direct product of two-dimensional local fields such that this product consists of one field provided that $x$ is smooth on $C$.  (Here a point  $x$  is a usual closed point.) Every two-dimensional local field which appear here is isomorphic to a field of iterated Laurent series $k'((u))((t))$, where a field $k'$ is a finite extension of $k$.

Besides, instead of the homomorphism~\eqref{deg} we have now a canonical central extension (see more  in Section~\ref{1-ext} below):
$$
0 \lrto \Z \lrto \widetilde{GL_n(\da_X)}  \lrto GL_n(\da_X)  \lrto 1 \, \mbox{.}
$$

The goal of this paper is to connect this central extension with the Riemann-Roch theorem for a locally free sheaf of $\oo_X$-modules of  rank $n$ on $X$ using the transition matrices for this sheaf, where these transition matrices are obtained from  the bases of
the completions of the stalks of the sheaf at scheme points of $X$.

From this central extension one obtains canonically another   central extension $\widehat{GL_n(\da_X)}$  (see Section~\ref{2-ext}) of $GL_n(\da_X)$ by $\Z$ such that by transition matrices $\alpha_{ij} \in GL_n(\da_X)$ ($i \ne j$ are from $\{0,1,2 \}$)
for a rank $n$ locally free sheaf of $\oo_X$-modules on $X$  and by the central extension $\widehat{GL_n(\da_X)}$  it is possible to obtain   the second Chern number $c_2$ of this sheaf. This was done in~\cite{Osip3},
see also Remark~\ref{c2}. Besides,  transition matrices $\alpha_{ij}$ are analogs for a locally free sheaf of $\oo_X$-modules on the surface $X$ of a transition matrix $\gamma_{01}$
for a locally free sheaf of $\oo_S$-modules on the curve $S$, see the reasoning above. (We have omitted here the indication on the sheaf in the notation for transition matrices.)

To solve   the above tasks we use  canonical lifts  of  transition matrices $\alpha_{ij}$ of a sheaf from the group $GL_n(\da_X)$ to the groups $\widetilde{GL_n(\da_X)}$ and $\widehat{GL_n(\da_X)}$.

Thus we obtain the {\em local (or adelic) decomposition} for the differences of Euler characteristics
of a rank $n$ locally free sheaf of $\oo_X$-modules   and the sheaf $\oo_X^n$, see Remark~\ref{lo-pr-2}.

We can say also that the main ingredient of this paper is the calculation of the integer
$\widetilde{\alpha_{02}} \cdot \widetilde{\alpha_{21}}   \cdot \widetilde{\alpha_{10}}$, where $\widetilde{\alpha_{ij}}$ is the canonical lift of
$\alpha_{ij}$ from $GL_n(\da_X)$ to  $\widetilde{GL_n(\da_X)}$. This integer does not depend on the choice of transition matrices $\alpha_{ij}$.

We do this calculation in two ways. The first way leads to Theorem~\ref{Th1} and uses adelic complexes  for rank $n$  locally free sheaves of $\oo_X$-modules  on $X$, see Proposition~\ref{Euler}. For the second way we suppose that a basic field $k$ is perfect, and we use the ``self-duality'' of the adelic space $\da_X$ based on reciprocity laws on $X$ for residues of differential two-forms on two-dimensional local fields introduced and studied in~\cite{Par}. This another way leads in Theorem~\ref{Th2} to  an answer,
which uses also another invariants of a sheaf   and $X$. The comparison of these two answers (after Theorems~\ref{Th1} and~\ref{Th2}) gives the Riemann-Roch theorem for a rank $n$ locally free sheaf of $\oo_X$-modules  on $X$ (without the Noether formula).

We note that the relation of the Riemann-Roch theorem on an algebraic surface with the local constructions was also discussed in~\cite{BS,FT, Sh}, but without using an adelic ring on a surface.

The paper is organized as follows.

In Section~\ref{adeles-first} we very briefly recall on adeles on algebraic surfaces and also recall  the construction of
the central extension $\widetilde{GL_n(\da_X)}$ of the group $GL_n(\da_X)$ by the group $\Z$.

In Section~\ref{2-ext} we recall the construction   of
the central extension $\widehat{GL_n(\da_X)}$ of the group $GL_n(\da_X)$ by the group $\Z$. This central extension  is obtained from the central extension $\widetilde{GL_n(\da_X)}$. In Remark~\ref{group-cohom} we  compare these central extensions from the point of view of group cohomology.

In Section~\ref{special}   we construct special elements in $\Z$-torsors. These elements are related to the special $k$-vector subspaces and  subrings in $\da_X$. The subspaces and subrings  are connected with points and irreducible curves on $X$.

In Section~\ref{intersect} we discuss various formulas (formula~\eqref{inter-index} and Proposition~\ref{prop-inter-index}) for the intersection index of divisors on $X$
related with special elements constructed in Section~\ref{special} and with the central extension  $\widetilde{GL_1(\da_X)}$.

In Section~\ref{splittings} we construct canonical splittings of the central extension $\widetilde{GL_n(\da_X)}$ over special subgroups of the group $GL_n(\da_X)$. We  discuss the properties of these splittings. In Remark~\ref{analog} we recall the corresponding splittings and their properties for the central extension $\widehat{GL_n(\da_X)}$  from~\cite{Osip3}.

In Section~\ref{trivializations}  we introduce the transition matrices $\alpha_{ij} \in GL_n(\da_X)$ for a rank $n$ locally free sheaf $\E$ of $\oo_X$-modules  on $X$ and canonical lifts $\widetilde{\alpha_{ij}}$ of $\alpha_{ij}$ to $\widetilde{GL_n(\da_X)}$ by means of canonical splittings from Section~\ref{splittings}.

In Section~\ref{first-way} we calculate the integer
$$f_{\E} = \widetilde{\alpha_{02}} \cdot \widetilde{\alpha_{21}}   \cdot \widetilde{\alpha_{10}} =
(\chi(\E) - n \chi(\oo_X)) - 2 {\rm ch}_2(\E)  $$
in the first way, where the rational number ${\rm ch}_2(\E) = \frac{1}{2}c_1(\E)^2 -c_2(\E)$.

In Section~\ref{second-way}, when the field $k$ is perfect  we calculate the integer
$$
f_{\E} = -\frac{1}{2} K \cdot c_1(\E) - {\rm ch}_2 (\E)
$$
in the second way, where $K \simeq \oo_X(\omega)$, $\omega \in \Omega^2_{k(X)/k}$, $\omega \ne 0$. We derive the Riemann-Roch theorem for $\E$.

I am grateful to A.\,N. Parshin for some discussions and comments. Actually, the motivation for this paper was discussion with him  that the  adelic technique on  an algebraic surface should imply the Riemann-Roch theorem.

\section{Constructions of central extensions of $GL_n(\da_X)$}  \label{sec-centr}

\subsection{Adelic ring on a surface and a first central extension}   \label{1-ext}

\label{adeles-first}

As we have already mentioned in \S~\ref{Intro} (Introduction), throughout the article,   $X$ is a smooth projective algebraic surface over a field  $k$. But for the constructions in \S~\ref{sec-centr}, it is not important that  $X$ is projective.

We recall (see, e.g., survey~\cite{Osip1} and also~\cite[\S~2.1]{Osip3}) that the adelic ring of $X$ is
$$
\da_X = \da_X = {\prod_{x \in C}}' K_{x,C}  \subset \prod_{x \in C} K_{x,C}
$$
and every ring $K_{x,C}= \prod_i K_i$ is the finite direct product of two-dimensional local fields $K_i$ such that every two-dimensional local field $K_i$ corresponds to the formal branch of $C$ at~$x$.

We define   a subring $\da_{12}  \subset \da_X$:
$$
\da_{12} = \da_X \cap  \prod_{x \in C}  \oo_{K_{x,C}}  \subset \prod_{x \in C} K_{x,C}    \, \mbox{,}
$$
where $\oo_{K_{x,C}} = \prod_i \oo_{K_i}$ is the finite direct product of discrete valuation rings $\oo_{K_i}$ of $K_i$. (If $K_i$ is isomorphic to $k_i((u))((t))$, then $\oo_{K_i}$ is isomorphic to $k_i((u))[[t]]$.)

For any locally linearly compact $k$-vector space (or, in other words, a Tate vector space) $U$ we recall canonical construction of $\dz$-torsor $\Dim(U)$ from~\cite{Kap}.
As a set, $\Dim(U)$ consists of all maps $d$ (which are called ``dimension theories'' in~\cite{Kap}) from the set of all open linearly compact $k$-vector subspaces of $U$ to $\Z$ with the property
$$
d(Z_2)= d(Z_1) + [Z_1 \mid Z_2] \, \mbox{,}
$$
where $Z_1, Z_2$ are any open linearly compact $k$-vector subspaces of $U$. The group $\dz$ acts on $\Dim(U)$  in the following way:
$$
(m + d)(Z_1)= d(Z_1) +m   \,  \mbox{,}  \quad \mbox{where} \quad m \in \Z  \, \mbox{.}
$$

For any exact sequence of locally linearly compact $k$-vector spaces
\begin{equation}  \label{ex-seq}
0 \lrto U_1  \stackrel{\phi_1}{\lrto} U_2 \stackrel{\phi_2}{\lrto} U_3 \lrto 0  \, \mbox{,}
\end{equation}
where $\phi_1, \phi_2$ are continuous maps and $\phi_1$ is a closed embedding, we have a canonical isomorphism:
\begin{gather}  \label{dim-iso}
\Dim(U_1)  \otimes_{\Z} \Dim(U_3)  \lrto \Dim(U_2) \, \mbox{,}  \\  d_1 \otimes d_3 \longmapsto d_2 \, \mbox{,} \quad
d_2(Z) = d_1(Z \cap U_1)  + d_3(\phi_2(U_1))  \notag  \, \mbox{,}
\end{gather}
where $Z$ is an open linearly compact $k$-vector subspace of $U_2$.

For any  locally linearly compact $k$-vector space $U$
we define the locally linearly compact  $k$-vector space $\check{U}$ as a $k$-vector subspace of the dual $k$-vector space $U^*$
in the following way:
$$
\check{U} = \bigcup_{W} W^{\perp}  \, \mbox{,}
$$
where $W$ runs over all open linearly compact  $k$-vector subspaces of $U$, the $k$-vector subspace $W^{\perp}  \subset U^*$ is the annihilator of $W$
in $U^*$, and $W^{\perp}$, which is the dual vector space to the discrete vector space $U/W$, is an open linearly compact $k$-vector subspace
of $\check{U}$. In other words, $\check{U}$ is the continuous dual space, i.e., it consists of all continuous linear functionals.
We have a canonical isomorphism:
\begin{equation}  \label{dual}
\Dim(U) \otimes_{\Z}  \Dim(\check{U})  \simeq \Z   \, \mbox{,}  \qquad d_1 \otimes d_2 \longmapsto d_1(Z) + d_2(Z^{\perp})  \, \mbox{,}
\end{equation}
where $Z \subset U$ is a linearly compact $k$-vector subspace, and the result does not depend on the choice of $Z$.

Let $D = \sum_{i} a_i C_i$ be any divisor on $X$, where $C_i$ are irreducible curves on $X$.  We denote
$$
\da_{12}(D)= \da_X \cap \prod_{x \in C} t_C^{-\nu_C(D)} \oo_{K_{x,C}}  \, \mbox{,}
$$
where the intersection is taken inside $\prod_{x \in C}  K_{x,C}$, and $t_C=0$ is an equation of an (irreducible) curve $C$ in some open  subset of $X$,
$\nu_C(D) $ equals $a_i$ when $C = C_i$ and zero otherwise. (The definition of $\da_{12}(D)$ does not depend on the choice of $t_C$.)

We note that
$$
\da_X = \mathop{\lim_{\lrto}}_{D_2} \mathop{\lim_{\longleftarrow}}_{D_1 \ge D_2} \da_{12}(D_2) / \da_{12}(D_1)    \mbox{,}
$$
and $\da_{12}(D_2) / \da_{12}(D_1)$ is a locally linearly compact $k$-vector space,
and for any divisors $D_1 \ge D_2 \ge D_3$ on $X$ the exact sequence
$$
0 \lrto    \da_{12}(D_2) / \da_{12}(D_1)     \lrto     \da_{12}(D_3) / \da_{12}(D_1)   \lrto \da_{12}(D_3) / \da_{12}(D_2)   \lrto 0
$$
is of type~\eqref{ex-seq},
see, e.g.,~\cite[\S~2.2.3]{Osip2}, \cite[\S~2.3]{Osip3}.

Let $n \ge 1$ be an integer.

\begin{defin} \label{lattice}
A $k$-vector subspace $E$  of $\da_X^n$
is called a {lattice} if and only if  there are divisors $D_1$ and $D_2$  on $X$ such that
$$
\da_{12}(D_1)^n  \subset E  \subset \da_{12}(D_2)^n
$$
and the image of $E$ in $\da_{12}(D_2)^n/ \da_{12}(D_1)^n$ is a closed $k$-vector subspace.
\end{defin}

If $E_1 \subset E_2$ are lattices,
then $E_2/ E_1$
is a locally linearly compact $k$-vector space with the quotient and induced topology from the locally linearly compact $k$-vector space
$ \da_{12}(D_2)^n/ \da_{12}(D_1)^n $. In this case we define a $\Z$-torsor
$$
\Dim(E_1 \mid E_2) = \Dim(E_2/E_1)   \, \mbox{.}
$$

 If $E_1 \subset E_2 \subset E_3$ are lattices, then an exact sequence
$$
0 \lrto E_2/E_1 \lrto E_3 / E_1 \lrto E_3 / E_2 \lrto 0
$$
is of type~\eqref{ex-seq}.

Now for arbitrary lattices $E_1$ and $E_2$  we define a  $\Z$-torsor
$$
\Dim(E_1 \mid E_2) = \mathop{\lim_{\lrto}}_E  \Hom\nolimits_{\Z} (\Dim(E_1/E), \Dim(E_2/E))  \, \mbox{,}
$$
where the direct limit is taken over all lattices $E \subset \da_X^n$ such that  $E \subset E_i$ for $i=1$ and $i=2$. Here  we use the following isomorphisms of $\Z$-torsors for lattices $E \supset E'$ and $i=1$, $i=2$:
$$
\Dim(E_i/E) \otimes_{\Z} \Dim(E/E')  \lrto \Dim(E_i/E') \, \mbox{,}
$$
so that  the transition maps in this direct limit are given as
$$
f \longmapsto f' \, \mbox{,}   \qquad f'(a \otimes c)= f'(a) \otimes c  \, \mbox{.}
$$

 Obviously, for any lattices $E_1, E_2, E_3$  we have  canonical isomorphism of $\Z$-torsors
 \begin{equation}  \label{dim-centr}
\Dim(E_1  \mid E_2) \otimes_{\Z} \Dim(E_2  \mid E_3) \lrto \Dim(E_1  \mid E_3)
\end{equation}
which satisfies the associativity diagram for four lattices.

It is not difficult to see that for any $g \in GL_n(\da_X)$ and any lattice $E$,
the $k$-vector subspace $g E$ is again a lattice. For any lattices $E_1, E_2$ we have an obvious isomorphism
$$
\Dim(E_1  \mid E_2 )  \lrto \Dim(g E_1   \mid g E_2 ) \, \mbox{,}
\qquad d \longmapsto g(d)
  \, \mbox{.}
$$

This all leads to the construction of the central extension
\begin{equation}  \label{1centr-ext}
0 \lrto \Z \lrto \widetilde{GL_n(\da_X)}  \stackrel{\Theta}{\lrto} GL_n(\da_X)  \lrto 1 \, \mbox{,}
\end{equation}
where $\widetilde{GL_n(\da_X)}$ consists of all pairs $(g, d)$, where $g \in GL_n(\da_X)$ and ${d \in \Dim(\da_{12}^n \mid g \da_{12}^n)}$.
The group operation and the map $\Theta$ are the following
\begin{equation}  \label{group law}
(g_1, d_1)(g_2, d_2)= (g_1 g_2, d_1 \otimes g_1(d_2)) \, \mbox{,} \qquad \Theta((g,d))=g  \, \mbox{.}
\end{equation}

\begin{nt} \label{Aut} \em
The above construction of the central extension $\widetilde{GL_n(\da_X)} $  is a special case of a more  general construction. Let $B$ be a $C_2$-space over $k$
(or, with a little bit more restrictions, a $2$-Tate vector space over $k$), see~\cite{Osip2}.  Then there is a canonical central extension
(see~\cite[\S~5.5, Remark~15]{OsipPar1})
$$
0 \lrto \Z \lrto \widetilde{\Aut\nolimits_{C_2}(B)_E}  \lrto \Aut\nolimits_{C_2}(B)  \lrto 1  \, \mbox{,}
$$
which depends on the choice of a lattice $E \subset B$, and where $\Aut_{C_2}(B)$ is the automorphism group of $B$ as an object of the category of $C_2$-spaces.
Now, $\da_X^n$ is a $C_2$-space over $k$ and there is an embedding of groups  $GL_n(\da_X)  \subset \Aut_{C_2}(\da_X^n)$, see~\cite{Osip2}. Then the central extension $\widetilde{GL_n(\da_X)} $ is the restriction of the central extension  $\widetilde{\Aut\nolimits_{C_2}(B)_E}$ under the last embedding of groups when $B$
is $\da_X^n$ and
${E}$ is~$ {\da_{12}^n}$.
\end{nt}

\subsection{A second central extension}
\label{2-ext}

From central extension~\eqref{1centr-ext} we will obtain another central extension (see also more in~\cite[\S~3.1]{Osip3}). We note that
$$
GL_n(\da_X) = SL_n(\da_X)  \rtimes \da_X^* \, \mbox{,}
$$
where the group of invertible elements  $\da_X^*$ of the ring $\da_X$ acts on $SL_n(\da_X)$ by conjugations, i.e. by inner automorphisms $h \mapsto aha^{-1}$, where
 $h \in SL_n(\da_X)$ and
 we embed $\da_X^*$ into $GL_n(\da_X)$ as   $a \mapsto \diag(a,1, \ldots, 1)$.
Now we define
$$
\widehat{GL_n(\da_X)} = \Theta^{-1}(SL_n(\da_X)) \rtimes \da_X^*  \, \mbox{,}
$$
where $\da_X^*$ acts on $\Theta^{-1}(SL_n(\da_X))$ by inner automorphisms in the group $\widetilde{GL_n(\da_X)}$ via lifting of elements from $\da_X^*$ to $\widetilde{GL_n(\da_X)}$, i.e.
$
a(g)= a' g a'^{-1} $, where $ g \in \Theta^{-1}(SL_n(\da_X))$ and $a' \in \widetilde{GL_n(\da_X)}$
is any element  such that $\Theta(a')= \diag(a,1, \ldots, 1)$.

Clearly, we obtain the central extension:
\begin{equation}  \label{2centr-ext}
0 \lrto \Z \lrto \widehat{GL_n(\da_X)}  {\lrto} GL_n(\da_X)  \lrto 1 \, \mbox{.}
\end{equation}
This central extension restricted to the subgroup $SL_n(\da_X)  \subset GL_n(\da_X)$ coincides
with the central extension  $\widetilde{GL_n(\da_X)}$ restricted to this subgroup. Besides, by construction, central extension~\eqref{2centr-ext}
canonically splits over the subgroup $\da_X^*$.

\begin{nt}
\label{group-cohom}
\em
We explain what  the transition from central extension~\eqref{1centr-ext} to central extension~\eqref{2centr-ext}  means
in terms of group  cohomology.

Recall (see~\cite[1.7. Construction]{BD}) that  a central extension $\hat{C}$ of a group $C = G \rtimes H$
by  a group $A$ is equivalent to the following data:
\begin{itemize}
\item[1)]   a central extension of $H$ by $A$;
\item[2)]   a central extension
$$
1 \lrto A \lrto \hat{G} \lrto G \lrto 1 \, \mbox{;}
$$
\item[3)] an action of the group $H$ on the group $\hat{G}$  such that this action  lifts the action of $H$ on $G$ and is trivial on $A$.
\end{itemize}
And an isomorphism of central extensions corresponds to an isomorphism of data.

Using this description and that $H$ is a subgroup in $C$ and $H$ is a quotient group of $C$,
we obtain that
\begin{equation}  \label{transition}
H^2(C, A) = H^2(H,A)  \oplus T   \, \mbox{,}
\end{equation}
and that  there is an exact sequence
\begin{equation}  \label{trans}
0 \lrto H^1(H, \Hom(G,A) ) \lrto T \stackrel{\psi}{\lrto}  H^2(G,A)^H  \stackrel{\phi}{\lrto} H^2(H, \Hom(G,A))   \, \mbox{,}
\end{equation}
which  we will now explain. (Here the actions of $H$  on $\Hom(G,A)$ and on $H^2(G,A)$ are obtained in the usual way from  the action on $G$.)

We will use that for any central extension
\begin{equation}  \label{ce}
1 \lrto A \lrto \hat{G} \lrto G \lrto 1
\end{equation}
the group $\Hom(G,A)$ is canonically isomorphic  to the automorphism group of central extension~\eqref{ce}, i.e. it is  isomorphic to the group of automorphisms of the group $\hat{G}$ such that these automorphisms induce  identically action on the subgroup $A$ and on the quotient group $G$.

We consider any central extension $I$ from $H^2(G,A)^H$.
 We have a group $L(I)$ which consists of lifts of the actions of elements of   $H$
 to actions on the group corresponding to the central extension  $I$,
  with the trivial action on the subgroup $A$. The group $L(I)$
  is an extension (in general, non-central) of the group  $H$
  by the automorphism group  $\Hom(G,A)$ of the central extension  $I$.
  The isomorphism class of this non-central extension is  $\phi(I)$. If $\phi(I) =0 $,
  then the action of the whole group   $H$
  can be lifted to an action on the group corresponding to the central extension  $I$,
  with the trivial action on the subgroup $A$.

An element $t$ from $T$ is an isomorphism class of the following data: a central extension $\psi(t)$ together with a lift of the action of the group $H$ to an action on the group  corresponding to  this central extension, with the trivial action on the subgroup~$A$.

Correspondingly, for any cental extension $I$ from $\Ker \phi$  we have that $\psi^{-1}(I)$ is the set of
equivalence classes of group sections $s$ of the natural homomorphism $L(I)  \to H$.
 Group sections $s$ and $s'$ are equivalent if  $s = \sigma_a s'$, where $\sigma_a$ is an inner automorphism of the group  $L(I)$, given by an element  $a$ from the automorphism group  $\Hom(G,A)$  of the central extension  $I$.
  Hence we obtain that   $\psi^{-1}(I)$ is an $H^1(H, \Hom(G,A))$-torsor.

We note that another, computational construction and proof of exact sequence~\eqref{trans} and its continuation  to the right  was given in~\cite{Ta}.

\bigskip

Now the transition from central extension~\eqref{1centr-ext} to central extension~\eqref{2centr-ext} is the projection from $H^2(C, A)$
to $T$ in~\eqref{transition}.

We note that if a group $G$ is perfect, i.e. $[G,G]=G$, then $\Hom(G,A)=0$. Therefore, in this case  we obtain from~\eqref{transition}-\eqref{trans}
that
$$
H^2(C,A)= H^2(H,A) \oplus H^2(G,A)^H   \, \mbox{.}
$$
This condition  is satisfied, for example,  when $C =GL_n(F)$, $H= F^*$, $G = SL_n(F)$,
where $F$ is an infinite field (for example, $F$ is an $n$-dimensional local field). Then we have (cf.~\cite[\S~2.2]{Osip2'})
$$
H^2(SL_n(F), A)^{F^*}= \Hom(H_2(SL_n(F), \Z)_{F^*}, A)= \Hom(K_2(F), A)  \mbox{.}
$$
\end{nt}

\section{Specially constructed elements of $\Z$-torsors}
\label{special}

For an irreducible curve $C$ on $X$ let a field $K_C$ be the completion of the field $k(X)$ of rational functions on $X$ with respect to the discrete valuation given by $C$.

For a point $x$ of $X$ let $K_x= k(X) \cdot \hat{\oo}_{x,X}$ be a subring
of the fraction field $\Frac \hat{\oo}_{x,X}$, where $\hat{\oo}_{x,X}$
is the completion of the local ring of $x$ on $X$.

We have natural diagonal embeddings:
\begin{equation}  \label{diag-emb}
\prod_{C \subset X} K_C  \hookrightarrow \prod_{x \in C} K_{x,C}   \qquad \mbox{and} \qquad
\prod_{x \in X}  K_x  \hookrightarrow  \prod_{x \in C} K_{x,C} \, \mbox{.}
\end{equation}

There are the following subrings of the adelic ring $\da_X$:

$$
\da_{01} = \left(   \prod_{C \subset X} K_C  \right)   \cap \da_X    \qquad  \mbox{and}  \qquad
\da_{02} = \left(   \prod_{x \subset X} K_x  \right)   \cap \da_X   \, \mbox{,}
$$
where the intersection is taken inside the ring $\prod_{x \in C} K_{x,C}$.

Let $n \ge 1$ be an integer.

For lattices $E_1 \subset E_2$ of $\da_X^n$ (see Definition~\ref{lattice}) we define a $k$-vector  subspace
$$
\tilde{\mu}_{E_1, E_2} = (E_2 \cap \da_{02}^n)/ (E_1 \cap \da_{02}^n)  \subset E_2/ E_1  \, \mbox{.}
$$
It is easy to see that for any divisors $D_1 \le D_2$ on $X$
the $k$-vector subspace
$${ (A_{12}(D_2) \cap \da_{02}^n)/ (A_{12}(D_1) \cap \da_{02}^n)  }$$
is an open linearly compact $k$-vector subspace of $\da_{12}(D_2)/ \da_{12}(D_1)$.
Hence $\tilde{\mu}_{E_1, E_2}$
 is an open linearly compact $k$-vector subspace of $E_2/E_1$.

Now for arbitrary lattices  $E_1$ and $E_2$ of $\da_X^n$ we introduce an element
$$\mu_{E_1, E_2}   \in \Dim(E_1 \mid E_2)$$
uniquely defined by the following two rules:
\begin{enumerate}
\item If $E_1 \subset E_2$, then $\mu_{E_1, E_2}$ is a ``dimension theory'' that equals $0$ on a $k$-vector subspace $\tilde{\mu}_{E_1, E_2}  \subset E_2/E_1$.
\item   For arbitrary lattices $F_1, F_2, F_3$ of $\da_X^n$ we have
$$
\mu_{F_1, F_2} \otimes \mu_{F_2, F_3}  = \mu_{F_1, F_3}
$$
with respect to  isomorphism~\eqref{dim-centr}.
\end{enumerate}

\begin{nt} \em
To construct an element $\mu_{E_1, E_2}  \subset \Dim(E_1 \mid E_2)$ it is not important that $X$ is a projective surface.
\end{nt}

Analogously, for lattices $E_1 \subset E_2$  of $\da_X^n$ we define a $k$-vector subspace
$$
\tilde{\nu}_{E_1, E_2} = (E_2 \cap \da_{01}^n)/ (E_1 \cap \da_{01}^n)  \subset E_2/ E_1  \, \mbox{.}
$$
The $k$-vector subspace $\tilde{\nu}_{E_1, E_2}  \subset E_2/E_1$ is a discrete subspace such that the quotient space of $E_2/E_1$ by this subspace
is a linearly compact space.  (This fact is first easy to see for $E_1 = \da_{12}(D_1)$, $E_2 = \da_{12}(D_2)$, where $D_1$ and $D_2$ are divisors on $X$. The latter follows, since for any projective curve $C$ on $X$ the field of rational functions $k(C)$ is a discrete subspace in the adelic space of $C$ and the quotient space of this adelic space by $k(C)$ is a linearly compact $k$-vector space, and this follows, for example,  from adelic complex of $C$ and that the cohomology spaces of coherent sheaves on $C$ are finite-dimensional $k$-vector spaces.)

Now for arbitrary lattices  $E_1$ and $E_2$ of $\da_X^n$ we introduce an element
$$\nu_{E_1, E_2}   \in \Dim(E_1 \mid E_2)$$
uniquely defined by the following two rules.
\begin{enumerate}
\item If $E_1 \subset E_2$, then $\nu_{E_1, E_2}  \in \Dim(E_2/ E_1)$
is defined from exact sequence
$$
0 \lrto   \tilde{\nu}_{E_1, E_2}
  \lrto E_2/ E_1  \lrto (E_2/E_1)/ \tilde{\nu}_{E_1, E_2}  \lrto 0
$$
where the first non-zero term is a discrete space and the last non-zero term is a linearly compact space, and by~\eqref{dim-centr} there is a canonical isomorphism
$$
 \Dim(\tilde{\nu}_{E_1, E_2}) \otimes_{\Z} \Dim(  (E_2/E_1)/  \tilde{\nu}_{E_1, E_2} ) \lrto \Dim(E_2/ E_1) \, \mbox{,}
$$
now $\nu_{E_1, E_2} = \nu_1 \otimes \nu_2$, where $\nu_1 \in \Dim( \tilde{\nu}_{E_1, E_2}  )$  is the ``dimension theory'' that equals $0$
on the zero-subspace of the discrete space $\tilde{\nu}_{E_1, E_2}$ and $\nu_2 \in  \Dim(  (E_2/E_1)/ \tilde{\nu}_{E_1, E_2}  )$
is the ``dimension theory'' that equals $0$ on the whole linearly compact $k$-vector space $(E_2/E_1)/ \tilde{\nu}_{E_1, E_2}$.
\item   For arbitrary lattices $F_1, F_2, F_3$ of $\da_X^n$ we have
$$
\nu_{F_1, F_2} \otimes \nu_{F_2, F_3}  = \nu_{F_1, F_3}
$$
with respect to  isomorphism~\eqref{dim-centr}.
\end{enumerate}

\begin{nt} \em
To construct an element $\nu_{E_1, E_2}  \subset \Dim(E_1 \mid E_2)$ it is  important that $X$ is a projective surface.
\end{nt}

\medskip

We note that any rank $n$ locally free subsheaf of $\oo_X$-modules $\E  \subset k(X)^n$  on $X$ gives a {\em lattice}  $\da_{12}(\E)  \subset  \da_X^n$
(see~\cite[\S~2.2.3]{Osip2}) which generalizes
the case $\da_{12} (\oo_X^n)=\da_{12}^n$
and the case $\da_{12}(\oo_X(D))=\da_{12}(D)$ for $n=1$, where $D$ is a divisor on $X$.

\begin{prop}  \label{Euler}
For any rank $n$ locally free subsheaves of $\oo_X$-modules $\f  $ and  $\g $ of the constant sheaf $k(X)^n$ on $X$ we have
\begin{equation}  \label{Ech}
\nu_{\da_{12}(\f), \da_{12}(\g)}  - \mu_{\da_{12}(\f), \da_{12}(\g)} = \chi(\g) - \chi(\f)  \, \mbox{,}
\end{equation}
where the substraction in the left hand side of the formula makes sense, because it is applied to the elements of the $\Z$-torsor
$ \Dim(\da_{12}(\f) \mid \da_{12}(\g))  $, and for any Zariski sheaf $\E$ on $X$
$$
\chi(\E) = H^0(X, \E)  - H^1(X, \E)
$$
is its Euler characteristic on $X$.
\end{prop}
\begin{proof}
From the properties of the left hand side and the right hand side of formula~\eqref{Ech} we see that it is enough to suppose that
$\f \subset \g$.
For any quasicoherent sheaf $\E$ of $\oo_X$-modules on $X$ there is the adelic complex $\ad_X(\E)$ which has  non-zero terms only in degrees $0,1,2$
and $H^i(X, \E)= H^i(\ad_X(\E))$ (see, e.g.,~\cite{Osip1}).

We have canonical embedding  of $\ad_X(\f)$ to $\ad_X(\g)$ such that the quotient complex looks as following:
\begin{gather*}
\tilde{\nu}_{\da_{12}(\f), \da_{12}(\g)}  \oplus \tilde{\mu}_{\da_{12}(\f), \da_{12}(\g)}  \lrto \da_{12}(\g)/ \da_{12}(\f) \\
x \oplus y   \longmapsto x +y \, \mbox{.}
\end{gather*}

Hence we obtain formula~\eqref{Ech}, since the Euler characteristic of the latest complex equals to $\chi(\ad(\g) ) - \chi(\ad(\f))$.

\end{proof}

\section{Intersection index of divisors}  \label{intersect}

Now we recall how it is possible to obtain the intersection index of divisors on $X$ by means of central extension~\eqref{1centr-ext}  when $n=1$:
$$
0 \lrto \Z \lrto \widetilde{GL_1(\da_X)}  \stackrel{\Theta}{\lrto} \da_X^*  \lrto 1 \, \mbox{.}
$$

We consider a bimultiplicative and anti-symmetric map
$$
\langle \cdot , \cdot \rangle \; :  \;  \da_X^*  \times \da_X^*  \lrto \Z
$$
given for any $x, y \in \da_X^*$ as following
$$
\langle x, y \rangle = [x', y']= x' y' {x'}^{-1} {y'}^{-1}  \, \mbox{,}
$$
where $x', y' \in \widetilde{GL_1(\da_X)}$ are any elements such that $\Theta(x')=x$ and $\Theta(y')=y$. This definition does not depend on the choice of $x'$ and $y'$.

We consider a divisor $D$ on $X$.

For any irreducible curve $C$ on $X$ let $j_C^D \in K_C^*$ be a local equation  of the restriction of $D$ to $\Spec \oo_{K_C}$ under the natural morphism
$\Spec \oo_{K_C} \to X$, where $\oo_{K_C}$ is the discrete valuation ring of $K_C$.  Note that $j_C^D \cdot v$, where $v \in \oo_{K_C}^*$, is  again such a local equation.
Then under the first diagonal embedding from~\eqref{diag-emb} we obtain that the collection $j_{1,D} = \prod\limits_C j_C^D$, where $C$ runs over all irreducible curves $C$ of $X$, belongs to $\da_{01}^*$.

For any point $x$ on $X$
let $j_x^D \in K_x^*$ be a local equation  of the restriction of $D$ to $\Spec \hat{\oo}_{x,X}$ under the natural morphism
$\Spec \hat{\oo}_{x,X} \to X$. Note that $j_x^D \cdot w$, where $w \in \hat{\oo}_{x,X}^*$, is  again such a local equation.
Then under the second diagonal embedding from~\eqref{diag-emb} we obtain that the collection $j_{2,D}= \prod\limits_x j_x^D$, where $x$ runs over all points $x$ of $X$, belongs to $\da_{02}^*$.

For any divisors $S$ and $T$ on $X$ we have (see~\cite[Proposition~2]{Osip3} which is based on~\cite[\S~2.2]{Par2}) that
\begin{equation}  \label{inter-index}
\langle j_{2, S} ,   j_{1,T}   \rangle  = - (S,T)  \, \mbox{,}
\end{equation}
where $(S,T)  \in \Z$ is the intersection index of $S$ and $T$ on $X$.

Now we give another presentation for the intersection index of divisors based on formula~\eqref{inter-index} and special lifts of elements as in Section~\ref{special}.

\begin{prop}
\label{prop-inter-index}
For any divisors $S$ and $T$ on $X$ we have
$$
(S, T)  = (\mu_{\da_{12}, \da_{12}(-T)} \otimes \nu_{\da_{12}(-T), \da_{12}(-T -S)})  - (\nu_{\da_{12}, \da_{12}(-S)}  \otimes
 \mu_{\da_{12}(-S)  ,\da_{12}(-S-T)})  =
$$
$$
=  (\mu_{\da_{12}, \da_{12}(T)} \otimes \nu_{\da_{12}(T), \da_{12}(T +S)})  - (\nu_{\da_{12}, \da_{12}(S)}  \otimes
 \mu_{\da_{12}(S)  ,\da_{12}(S+T)}) \, \mbox{,}
$$
where the substraction in these formulas makes sense, because it is appllied to elements of the $\Z$-torsor
$$ \Dim(\da_{12} \mid \da_{12}(-S -T)) \qquad  \mbox{or} \qquad \Dim(\da_{12}  \mid \da_{12}(S +T))$$
 using isomorphism~\eqref{dim-centr}.
\end{prop}
\begin{proof}
The second equality follows from the first equality, because $(S, T)= (-S, -T)$.

We prove the first equality. We note that for any divisor $D$ on $X$ we have ${ j_{1,D } \cdot   j_{2,D }^{-1}  \in \da_{12}^* }$ and equality of $k$-vector subspaces in $\da_X$:
$$
 j_{1,D } \da_{12} =   j_{2,D }  \da_{12}  = \da_{12}(-D)  \, \mbox{.}
$$
Therefore we can take the following special lifts of elements $j_{1,T}$ and $j_{2,S}$ to $\widetilde{GL_1(\da_X)}$
to calculate $\langle  j_{1,T} , j_{2,S} \rangle$:
\begin{equation}   \label{lift}
 j_{1,T  }  \longmapsto (  j_{1,T  } , \mu_{\da_{12}, \da_{12}(-T)} )  \, \mbox{,}  \qquad
 j_{2,S   }  \longmapsto   (  j_{2,S  } , \nu_{\da_{12}, \da_{12}(-S)} )  \, \mbox{.}
\end{equation}

Immediately from constructions we have that
\begin{gather}  \label{spaces}
j_{1,T}    (\nu_{\da_{12}, \da_{12}(-S)} )  = \nu_{\da_{12}(-T), \da_{12}(-T-S)}
\quad
\mbox{and} \\
\label{spaces1}
j_{2,S  }  (\mu_{\da_{12}, \da_{12}(-T)} )  = \mu_{\da_{12}(-S), \da_{12}(-T-S)}  \, \mbox{.}
\end{gather}

Now from  formula~\eqref{inter-index}  we have the following equality in the group $\widetilde{GL_1(\da_X)}$:
$$
(  j_{1,T  } , \mu_{\da_{12}, \da_{12}(-T)} )(  j_{2,S  } , \nu_{\da_{12}, \da_{12}(-S)} )
= (  j_{2,S  } , \nu_{\da_{12}, \da_{12}(-S)} )(  j_{1,T  } , \mu_{\da_{12}, \da_{12}(-T)} ) (S,T)  \, \mbox{,}
$$
where we consider the intersection index $(S, T )$ as an element of the central subgroup  ${\Z \subset \widetilde{GL_1(\da_X)}}$.
Now, using definition of the group operation in $\widetilde{GL_1(\da_X)}$ and formulas~\eqref{spaces}-\eqref{spaces1}, we obtain the statement of the proposition.

\end{proof}

\begin{nt} \em
Analogous to~\eqref{lift}  lift of elements was considered in~\cite[\S~5]{OP} for divisors $S$ and $T = (\omega) -  S$, where $\omega \in \Omega^2_{k(X)/k}$.
\end{nt}

\section{Canonical splittings}
\label{splittings}

We construct  canonical splittings of central extensions~\eqref{1centr-ext} and~\eqref{2centr-ext}     over some subgroups of $GL_n(\da_X)$ and investigate the properties of these splittings.

We consider central extension~\eqref{1centr-ext}.
\begin{prop}  \label{split}
We consider central extension~\eqref{1centr-ext}.
\begin{enumerate}
\item  \label{it1} The map
\begin{equation}  \label{spl12}
GL_n(\da_{12})  \ni g \longmapsto (g, 0)  \in \widetilde{GL_n(\da_X)}  \, \mbox{,}
\end{equation}
where $0$ is the zero ``dimension theory'',
gives a splitting of this central extension over the subgroup $GL_n(\da_{12})$.
\item   \label{it2}  The map
\begin{equation}  \label{spl02}
GL_n(\da_{02})  \ni g \longmapsto (g, \mu_{\da_{12}^n, g \da_{12}^n})  \in \widetilde{GL_n(\da_X)}
\end{equation}
gives a splitting of this central extension over the subgroup $GL_n(\da_{02})$.
\item  \label{it3}  The map
\begin{equation}  \label{spl01}
GL_n(\da_{01})  \ni g \longmapsto (g, \nu_{\da_{12}^n, g \da_{12}^n})  \in \widetilde{GL_n(\da_X)}
\end{equation}
gives a splitting of this central extension over the subgroup $GL_n(\da_{01})$.
\item \label{it4} Splittings~\eqref{spl12} and~\eqref{spl01} coincide over the subgroup  $GL_n(\prod_C \oo_{K_C})$. Splittings~\eqref{spl12} and~\eqref{spl02}
coincide over the subgroup  $GL_n(\prod_x  \hat{\oo}_{x,X})$. Splittings~\eqref{spl01} and~\eqref{spl02}
coincide over the subgroup  $GL_n(k(X))$. (Here we use the diagonal embeddings of subgroups, see~\eqref{diag-emb}.)
\end{enumerate}
\end{prop}
\begin{nt} \em For the construction of splittings~\eqref{spl12}-\eqref{spl02} it is not important that $X$ is projective, but for the construction of splitting~\eqref{spl01} it is important.
\end{nt}
\begin{proof}

Item~\ref{it1} is evident, since $g \da_{12}^n = \da_{12}^n$  for any $g \in GL_n(\da_{12})$.

Item~\ref{it2} and~\ref{it3} follow from the equalities:
$$
g \tilde{\mu}_{E_1, E_2} = \tilde{\mu}_{g E_1, gE_2} \, \mbox{,}  \qquad h \tilde{\nu}_{E_1, E_2}= \tilde{\nu}_{hE_1, hE_2}  \, \mbox{,}
$$
where $E_1 \subset E_2$  are any lattices in $\da_X^n$, an element $g \in GL_n(\da_{02})$, an element ${h \in GL_n(\da_{01})}$.

In item~\ref{it4} the only non-evident statement is that splittings~\eqref{spl01} and~\eqref{spl02}
coincide over the subgroup  $GL_n(k(X))$. To prove this statement we note that  for any $g \in GL_n(k(X))$  we have
$ g \da_{12}^n= \da_{12}(g \oo_X^n) $ and
by Proposition~\ref{Euler}  we have
$$
\nu_{\da_{12}(\oo_X^n), \da_{12}(g \oo_X^n)}  - \mu_{\da_{12}(\oo_X^n), \da_{12}(g \oo_X^n)} = \chi(g\oo_X^n) - \chi( \oo_X^n)  \, \mbox{.}
$$
We note that the sheaf $g \oo_X^n$ is isomorphic to the sheaf $\oo_X^n$. Therefore we have ${\chi(\oo_X^n) = \chi(g \oo_X^n) }$. Hence
  $$\nu_{\da_{12}^n, g\da_{12}^n}  = \mu_{\da_{12}^n, g\da_{12}^n}$$
and two splittings coincide. (We note that we could apply Proposition~\ref{Euler} only for the case $n=1$, since
$GL_n(k(X)) = SL_n(k(X)) \rtimes k(X)^*$ and the group $SL_n(k(X))$ is perfect that implies that any two sections of any central extension of $SL_n(k(X))$ coincide.)

\end{proof}

Let $n = n_1 + n_2$, where $n_1$ and $n_2$ are positive integers. We consider a parabolic subgroup
$$
P_{n_1, n_2} = \left\{ \left(
\begin{matrix}
GL_{n_1}(\da_X)  & * \\
0 & GL_{n_2}(\da_X)
\end{matrix}
 \right) \right\}
\subset GL_n(\da_X)  \, \mbox{.}
$$
Let $p_i : GL_n(\da_X) \to P_{n_i}$, where $i=1$ or $i=2$, be natural homomorphisms.
\begin{prop}  \label{parabolic}
 The pullback of the central extension $\widetilde{GL_n(\da_X)}$  via the embedding ${P_{n_1, n_2}  \hookrightarrow GL_n(\da_X)}$
is isomorphic to the Baer sum of pullbacks $p_1^*(\widetilde{GL_{n_1}(\da_X)})$
and $p_2^*(\widetilde{GL_{n_2}(\da_X)}) $. Besides the splittings (over special subgroups) from Proposition~\ref{split}
are compatible with respect to this isomorphism.
\end{prop}
\begin{proof} Construction of isomorphism is directly obtained from exact triple
$$
0 \lrto \da_X^{n_1}  \lrto \da_X^n \lrto \da_X^{n_2}  \lrto 0 \, \mbox{,}
$$
the fact that the action of $P_{n_1, n_2}$ preserves this triple, from construction of the group $\widetilde{GL_n(\da_X)}$ and
from isomorphism~\eqref{dim-iso}.

The compatibility of sections follows from  exact triples
$$
0 \lrto \da_{01}^{n_1} \lrto \da_{01}^n \lrto \da_{01}^{n_2} \lrto 0  \qquad \mbox{and} \qquad
0 \lrto \da_{01}^{n_1} \lrto \da_{01}^n \lrto \da_{01}^{n_2} \to 0
$$
and that the corresponding groups $P_{n_1,n_2} \cap GL_n(\da_{01})$
or $P_{n_1,n_2} \cap GL_n(\da_{02})$ act on these triples.
\end{proof}

\begin{nt} \label{analog} \em
A direct analog of Proposition~\ref{split} is valid for central extension~\eqref{2centr-ext} (see~\cite[Prop.~4]{Osip3}), where we use that
$\widehat{GL_n(\da_X)} = \Theta^{-1}(SL_n(\da_X)) \rtimes \da_X^*$, and we take the splittings over the intersections of $SL_n(\da_X)$ with the corresponding subgroups that come from Proposition~\ref{split} and the identity splitting over the intersection of $\da_X^*$ with the corresponding subgroups.

The behaviour of the pullback of central extension~\eqref{2centr-ext} via the map ${P_{n_1, n_2}  \hookrightarrow GL_n(\da_X)}$
differs from what is described in Proposition~\ref{parabolic}   for central extension~\eqref{1centr-ext}, see~\cite[Prop.~3]{Osip3}. This is because
central extension~\eqref{2centr-ext} is connected with the second Chern number of  a rank $n$ locally free sheaf of $\oo_X$-modules, see~\cite[\S~3.3]{Osip3} and Remark~\ref{c2} below.
\end{nt}

\section{Triviliazations of locally free sheaves}
\label{trivializations}

We describe trivializations of locally free  sheaves of $\oo_X$-modules over  completions of local rings of scheme points of $X$.
We consider a point $x \in X$ as a closed scheme point of $X$, the generic point of an irreducible curve $C$  on $X$ as a non-closed point of $X$, and also consider the generic point of $X$.

Let $\E$ be a locally free sheaf of $\oo_X$-modules of rank $n$ on $X$.

For   a (closed) point $x$ of $X$, the completion of the stalk of $\E$ at $x$ is a free $\hat{\oo}_{x,X}$-module. Let $e_x$ be a a basis of this module. We will also call such a basis a basis of $\E$ restricted to $\Spec \hat{\oo}_{x,X}$.

For  an irreducible curve  $C$ on $X$, the completion of the stalk of $\E$ at the generic point of $C$ is a free $\oo_{K_C}$-module.
Let $e_C$  be a basis of this module.  We will also call such  a basis a basis   of $\E$ restricted  to $\Spec \oo_{K_C}$.

The stalk of $\E$ at the generic point of $X$ is a $k(X)$-vector space.
Let $e_0$  be a basis of this vector space. We will call such a basis a basis   of $\E$ restricted  to $\Spec k(X)$.

By embedding of the completions of stalks of  $\E$ at scheme points of  $X$
to the tensor products  (over the local rings of the points)
of the stalks of  $\E$  and the corresponding rings: $K_x$, $K_C$ or $K_{x,C}$, we obtain
 {\em transition matrices}
$$\alpha_{02, x}  \in GL_n(K_x) \, \mbox{,} \qquad  \alpha_{01,C}  \in GL_n(K_C)  \, \mbox{,}  \qquad  \alpha_{21,x,C} \in GL_n(\oo_{K_{x,C}})$$
defined by the following equalities
$$
e_0 = \alpha_{02,x} \, e_x  \, \mbox{,} \qquad e_0 = \alpha_{01,C} \, e_C \, \mbox{,}  \qquad e_x = \alpha_{21,x,C} \, e_C  \, \mbox{.}
$$

Denote collections of matrices $\alpha_{02} =\prod\limits_x \alpha_{02,x}$, where $x$ runs over all closed points of $X$,
${ \alpha_{01}  = \prod\limits_C  \alpha_{01,C}    }$, where $C$
runs over all irreducible curves $C$ on $X$,  $ \alpha_{21}  = \prod\limits_{x \in C}  \alpha_{21,x,C}    $, where $x \in C$ runs over all pairs $x \in C$ with $x$ a closed point on an irreducible curve $C$ on $X$.

Define now $\alpha_{ij} = \alpha_{ji}^{-1}$, where $i \ne j$ from the set $\{0, 1, 2  \}$. Then for any $i \ne j \ne k$ from the set
$\{ 0,1,2\}$ we have an equality ({\em cocycle identity})  in  $GL_n \left( \prod\limits_{x \in C} K_{x,C} \right)$, where we use diagonal embeddings~\eqref{diag-emb},
\begin{equation}  \label{cocycle}
\alpha_{ij}  \cdot \alpha_{jk} \cdot  \alpha_{ki}  =1  \, \mbox{.}
\end{equation}

Note that if we change the chosen bases, then the matrices will change in the following way
\begin{gather}
\alpha_{02}  \longmapsto \alpha_{0} \cdot \alpha_{02} \cdot \alpha_{2}^{-1} \, \mbox{,}  \qquad
\alpha_{01}  \longmapsto \alpha_{0} \cdot \alpha_{01} \cdot \alpha_{1}^{-1}  \, \mbox{,}  \qquad
\alpha_{21}   \longmapsto  \alpha_{2} \cdot \alpha_{21}  \cdot \alpha_{1}^{-1}  \, \mbox{,}   \label{change} \\
\mbox{where} \qquad
\alpha_{0} \in GL_n(k(X)) \subset GL_n(\da_X)   \,  \mbox{,} \qquad
\alpha_{1} \in GL_n \left(\prod_C \oo_{K_C} \right)  \subset GL_n(\da_X)   \, \mbox{,}   \nonumber \\
\alpha_{2}  \in GL_n \left(\prod_x  \hat{\oo}_{x,X} \right)  \subset GL_n(\da_X)  \, \mbox{.}  \nonumber
\end{gather}

Using diagonal embeddings~\eqref{diag-emb} and by means of non-difficult reasonings, one can show  (see~\cite[\S~3.3]{Osip3}) that
\begin{equation}  \label{subgroups}
\alpha_{02}  \in GL_n(\da_{02}) \, \mbox{,}  \qquad  \alpha_{01} \in GL_n(\da_{01})  \, \mbox{,}  \qquad  \alpha_{21}  \in GL_n(\da_{12})   \, \mbox{.}
\end{equation}

We will use notation $\alpha_{ij, \E}$ instead of $\alpha_{ij}$ when it is not clear from the context which bundle this notation is associated with.
\begin{nt} \label{trivial}  \em
 If $n=1$, then $\E = \oo_X(D)$ for some divisor $D$ on $X$. We have that (see Section~\ref{intersect})
$$
j_{2, D}  = \alpha_{02, \E} \quad \mbox{and}  \quad  j_{1,D}    = \alpha_{01, \E}  \, \mbox{.}
$$
Therefore from formula~\eqref{inter-index} we obtain that for invertible sheaves ${\mathcal F}$ and $\mathcal G$ on $X$, their intersection index
$({\mathcal F}, \mathcal G)$ is
\begin{equation}  \label{inter-index-new}
({\mathcal F}, {\mathcal G}) = \langle \alpha_{01, {\mathcal G}} , \alpha_{02, {\mathcal F}}   \rangle =
 \langle  \alpha_{02, {\mathcal F}},  \alpha_{10, {\mathcal G}}   \rangle =
\langle  \alpha_{12, {\mathcal F}},  \alpha_{10, {\mathcal G}}   \rangle =
\langle \alpha_{12, {\mathcal F}} , \alpha_{20, {\mathcal G}}   \rangle   \, \mbox{,}
\end{equation}
where we used also splittings of central extension~\eqref{1centr-ext} over certain subgroups from Proposition~\eqref{split}, whence it follows that  $\langle \cdot, \cdot \rangle$ restricted to these subgroups equals $0$.
\end{nt}

\begin{defin}  \label{number}
Let $\E$ be a locally free sheaf of $\oo_X$-modules of rank $n$ on $X$.
For any $i \ne j$ from $\{0, 1,2 \}$ let $\widetilde{\alpha_{ij}}$ be the canonical lift of $\alpha_{ij}$ for $\E$ to $\widetilde{GL_n(\da_X)}$ by splitting over the corresponding subgroup, see Proposition~\ref{split} and formula~\eqref{subgroups}.  We { define}
$$
f_{\E}  = \widetilde{\alpha_{02}}  \cdot \widetilde{\alpha_{21}} \cdot  \widetilde{\alpha_{10}}     \in \Z  \, \mbox{.}
$$
\end{defin}
This definition is correct, because from formula~\eqref{cocycle} it follows that $f_E \in \Z$, and by Proposition~\ref{split} and formula~\eqref{change}, the integer $f_{\E}$ depends only on $\E$, i.e., this integer does not depend on the choice of bases $e_0$, $\{ e_x \}$, $\{ e_C \}$
for~$\E$.

\smallskip

As we have already mentioned before, {\em the goal of this paper} is to calculate $f_E$ and to relate the calculations of  this integer, given by two different ways, to the Riemann-Roch theorem for $\E$ on $X$ (without the Noether formula for $\oo_X$).

\begin{nt}  \label{c2}  \em
For central extension~\eqref{2centr-ext}
$$
0 \lrto \Z \lrto \widehat{GL_n(\da_X)}  {\lrto} GL_n(\da_X)  \lrto 1
$$
let $\widehat{\alpha_{02}}, \widehat{\alpha_{21}}, \widehat{\alpha_{10}}$ be the corresponding canonical lifts  of
elements $\alpha_{02}, \alpha_{21},  \alpha_{10} $ for $\E$  to $\widehat{GL_n(\da_X)}$  by corresponding splittings over special subgroups of $GL_n(\da_X)$, see Remark~\ref{analog}.
Then the following integer does not depend on the choice of bases $e_0$, $\{ e_x \}$, $\{ e_C \}$ and by~\cite[Theorem 1]{Osip3}   we have
$$
\widehat{\alpha_{02}} \cdot \widehat{\alpha_{21}} \cdot  \widehat{\alpha_{10}} = c_2(\E) \, \mbox{.}
$$
\end{nt}

\section{The first way to calculate $f_{\E}$}  \label{first-way}
Let $\E$ be a locally free sheaf of $\oo_X$-modules of rank $n$ on $X$. We calculate the integer $f_{\E}$ (see Definition~\ref{number})
in the first way.

\begin{prop} \label{tr-bl} The following properties are satisfied.
\begin{enumerate}
\item  \label{item1} Consider an exact triple of finite rank locally free sheaves of $\oo_X$-modules:
$$
0 \lrto {\E}_1 \lrto {\E}  \lrto {\E_2}  \lrto 0  \, \mbox{.}
$$
We have $f_{\E}  = f_{\E_1}  + f_{\E_2}$.
\item  \label{item2}
Let $\pi : Y \lrto X$ be the blow-up of a point on $X$. We have $f_{\E} = f_{\pi^*(\E)}$.
\end{enumerate}
\end{prop}
\begin{proof}
1. The integers $f_{\E_i}$, where $i \in \{1,2,3 \}$, do not depend on the choice of bases of $\E_i$ at completions of local rings of scheme points of $X$. Therefore we can choose, first, the bases for $\E_1$ and then complete them to the bases of $\E$. Hence
all the transition matrices $\alpha_{ij}$ belong to the subgroup $P_{n_1, n_2}  \subset GL_n(\da_X)$, where $n_1$ or $n_2$ are the corresponding ranks of  $\E_1$ or $\E_2$.
Now we apply  Proposition~\ref{parabolic}.

2. (Cf. the proof of~\cite[Theorem~1]{Osip3}.) Let $\pi$ be the blow-up of a point $x \in X$, and $\pi^{-1}(x)= R$. We note that
\begin{equation}  \label{decomp}
\da_Y = \da_X \times {\prod_{p \in R}}'  K_{p,R}  \, \mbox{.}
\end{equation}
Since $f_{\E}$  and $f_{\pi^*(\E)}$ do not depend on the choice of bases for $\E$ and $\pi^*(\E)$, we choose the special bases. Fix a trivialization of $\E$ on an  open neighbourhood of $x$ on $X$. This trivialization induces the bases $e_0$, $e_x$, $e_R$ and $e_p$, where $p \in R$. We
have a canonical isomorphism
$$
\delta \; : \;  \widetilde{GL_n(\da_X)}  \lrto \gamma^*(\widetilde{GL_n(\da_Y)})  \, \mbox{,}
$$
where the embedding $\gamma : GL_n(\da_X)  \hookrightarrow GL_n(\da_Y)$  is induced by decomposition~\eqref{decomp}. Now for any $i \ne j $ from
 $\{0, 1, 2 \}$ we have
$$
\gamma(\alpha_{ij, \E})= \alpha_{ij, \pi^*(\E)}  \,  \qquad  \mbox{and} \qquad
\delta(\widetilde{\alpha_{ij, \E}}) = \widetilde{\alpha_{ij, \pi^*(\E)}}  \, \mbox{,}
$$
where we consider $\widetilde{\alpha_{ij, \pi^*(\E)}}$ as elements of the group $\gamma^*(\widetilde{GL_n(\da_Y)})$. This implies the statement.
\end{proof}

\medskip

In the sequel, we denote also the intersection index $\E \cdot \ff = (\E, \ff)  \in \Z$ or $c_1(\E)^2= \E \cdot \E = (\E, \E) \in \Z$,
where $\E$ and $\ff$ are invertible sheaves on $X$.

\begin{Th}  \label{Th1}
Let $\E$ be a locally free sheaf of $\oo_X$-modules of rank $n$ on a smooth projective surface $X$ over a field $k$. We have
$$
f_{\E} = \left(\chi(\E)  - n \chi(\oo_X) \right) -c_1(\E)^2  + 2 c_2(\E) = \left(\chi(\E)  - n \chi(\oo_X) \right) - 2 \, {\rm ch}_2(\E)  \, \mbox{.}
$$
\end{Th}
\begin{proof}
The second equality is just reformulation of the first equality. Therefore we prove the first equality.

The left hand side and the right hand side of the equality are additive with respect to the short exact sequences of locally free sheaves of $\oo_X$-modules and are preserved by blow-ups of points (see Proposition~\ref{tr-bl} for the left hand side of the equality,
in the right hand side of the equality it is clear for  ${\rm ch}_2$, and for the  difference of Euler characteristics this follows from the simple case  $\chi(\ff) - \chi(\E) = \chi(\ff / \E)$, where  $\ff$  is  a locally free sheaf of   $\oo_X$-modules of rank $n$, $\E \subset \ff$ and the sheaf $\ff$ coincides with the sheaf  $\E$ in a neighbourhood of a point which is blew up). Therefore, by the splitting principle for locally free sheaves on smooth surfaces (cf. the proof of~\cite[Theorem~1]{Osip3}), it is enough to prove the equality for the case $n=1$.

So, we suppose $n=1$.
The chosen basis $e_0$ of $\E$ at the generic point of $X$ gives the embedding of $\E$ to the constant sheaf $k(X)$ on $X$. (Therefore $\E = \oo_X(D)$
for some divisor $D$ on $X$.) We fix also the other bases for $\E$ and hence the transitions matrices  $\alpha_{ij}$ for $\E$, where $i \ne j $ from
$\{ 1,2 \}$, as in Section~\ref{trivializations}.

We have to prove that
\begin{equation}  \label{rest}
f_{\E} = \left(\chi(\E)  - n \chi(\oo_X) \right)  - (\E, \E)  \, \mbox{.}
\end{equation}

 From formula~\eqref{inter-index-new}, which describes the intersection index of invertible sheaves as the commutator of lifting of corresponding elements to $\widetilde{GL_1(\da_X)}$, we have
\begin{equation}  \label{first-eq}
\widetilde{\alpha_{10}} \cdot \widetilde{\alpha_{02}}=( -(\E, \E)) \cdot \widetilde{\alpha_{02}} \cdot \widetilde{\alpha_{10}}  \, \mbox{,}
\end{equation}
where  we consider  $ -(\E, \E) $ as an element of the central subgroup  $ \Z \subset \widetilde{GL_1(\da_X)}$.

Besides, using conjugation, we obtain
$$
f_{\E}= \widetilde{\alpha_{02}}  \cdot \widetilde{\alpha_{21}} \cdot  \widetilde{\alpha_{10}} =
\widetilde{\alpha_{10}} \cdot
\widetilde{\alpha_{02}}  \cdot \widetilde{\alpha_{21}} \cdot  \widetilde{\alpha_{10}} \cdot \widetilde{\alpha_{10}}^{-1}
= \widetilde{\alpha_{10}} \cdot
\widetilde{\alpha_{02}}  \cdot \widetilde{\alpha_{21}}  \, \mbox{.}
$$
Thus, we have
\begin{equation}  \label{second-eq}
\widetilde{\alpha_{10}} \cdot
\widetilde{\alpha_{02}}  =   f_{\E} \cdot \widetilde{\alpha_{12}}  \, \mbox{,}
\end{equation}
where we consider $ f_{\E}$ as an element  of the central subgroup $\Z \subset \widetilde{GL_1(\da_X)}$.

From formulas~\eqref{first-eq}-\eqref{second-eq} we have that formula~\eqref{rest} will follow from the following lemma.

\begin{lemma}  \label{lem1}
Let $\E$ be an invertible sheaf on $X$. Then we have
$$
\widetilde{\alpha_{02}} \cdot \widetilde{\alpha_{10}}  =
(\nu_{\da_{12}, \da_{12}(\E)}  - \mu_{\da_{12}, \da_{12}(\E)})
\cdot \widetilde{\alpha_{12}} =
\left(\chi(\E)  -  \chi(\oo_X) \right)  \cdot \widetilde{\alpha_{12}}  \, \mbox{,}
$$
where we consider ${\nu_{\da_{12}, \da_{12}(\E)}  - \mu_{\da_{12}, \da_{12}(\E)}}$   and  ${\chi(\E)  - n \chi(\oo_X)}$ as  elements  of the central subgroup $\Z \subset \widetilde{GL_1(\da_X)}$, and the chosen and fixed basis $e_0$ of $\E$ at the generic point of $X$ gives the embedding of $\E$
to the constant sheaf $k(X)$ on $X$.
\end{lemma}

We note that since $\alpha_{02}$ and $\alpha_{10}$  are from $\da_X^*$, we have that ${\alpha_{02}  \cdot \alpha_{10} = \alpha_{10}  \cdot \alpha_{02} = \alpha_{12}}$.

We prove now this lemma.
The second equality in the statement of the lemma follows immediately  from the first equality and Proposition~\ref{Euler}. Therefore we prove the first equality.

Let $c \in \Z \subset \widetilde{GL_1(\da_X)} $  such that
$\widetilde{\alpha_{02}} \cdot \widetilde{\alpha_{10}}  =  c \cdot \widetilde{\alpha_{12}} $. Then we have
\begin{equation}  \label{form-c}
\widetilde{\alpha_{10}} = c \cdot \widetilde{\alpha_{20}} \cdot \widetilde{\alpha_{12}} \, \mbox{.}
\end{equation}
We note that $\alpha_{12} \da_{12} = \da_{12}$. Therefore $\widetilde{\alpha_{12}} = (\alpha_{12}, 0)$,
where ${0 \in \Z = \Dim(\da_{12}  \mid \da_{12})}$.  Hence, product with $\widetilde{\alpha_{12}}$ in the right hand side of formula~\eqref{form-c}
does not affect the ``dimension theory''.

Hence (see also Remark~\ref{trivial}, where transition matrices for an invertible sheaf are explicitly given) and by direct calculation we obtain that
$$
c = \nu_{\da_{12}, \da_{12}(\E)}  - \mu_{\da_{12}, \da_{12}(\E)}  \, \mbox{.}
$$
 This finishes the proof of the lemma and, consequently,  the proof of the theorem.

\end{proof}

\begin{nt}  \label{lo-pr-2} \em
From Theorem~\ref{Th1}, Remark~\ref{c2}  and formula~\eqref{inter-index-new} we obtain the following ``local (adelic) decomposition'' for the difference of Euler characteristics for a rank $n$ locally free sheaf $\E$ of $\oo_X$-modules  and the sheaf $\oo_X^n$, using central extensions~\eqref{1centr-ext} and~\eqref{2centr-ext} and transition matrices for $\E$:
$$
\chi(\E)  - n \chi(\oo_X)=  \widetilde{\alpha_{02}}  \cdot \widetilde{\alpha_{21}} \cdot  \widetilde{\alpha_{10}}
- 2  \widehat{\alpha_{02}}  \cdot \widehat{\alpha_{21}} \cdot  \widehat{\alpha_{10}}
+ \langle \det(\alpha_{02}), \det(\alpha_{10})     \rangle   \, \mbox{.}
$$
This formula generalizes formulas~\eqref{diff-1}-\eqref{deg} from \S~\ref{Intro} (Introduction) from the case of smooth projective algebraic curves to the case of smooth projective algebraic surfaces.
\end{nt}

\section{The second way to calculate $f_{\E}$ and the Riemann-Roch theorem}  \label{second-way}

There is another way to calculate $f_{\E}$ for a locally free sheaf $\E$ of $\oo_X$-modules of rank $n$ on $X$. We will do this way in this section.  This another way leads to an answer,
which uses also another invariants of $\E$ and $X$,
 see Theorem~\ref{Th2} below. And the comparison of this answer with the answer obtained in Theorem~\ref{Th1} immediately gives us the Riemann-Roch theorem for $\E$ on $X$ (without the Noether formula), see Corollary~\ref{Cor1}.

{\em In the sequel, we suppose that the basic field $k$ is perfect} (it is important for the theory of two-dimensional residues used below).
The idea for the new calculation of $f_{E}$ is to use the fact that $\da_X$ is self-dual as a $C_2$-space over $k$ (or a $2$-Tate vector space) and to make some calculations as on the ``dual side''.

More exactly, the self-duality of $\da_X$ is given by the following pairing (cf.~\cite[\S~2]{OP}). We fix
$\omega \in \Omega^2_{k(X)/k}$ such that $\omega \ne 0$. Then by $\omega $ we construct a bilinear symmetric non-degenerate  pairing by means of the residues on  two-dimensional local fields (see more about residues on $n$-dimensional local fields in~\cite{Par, Y})
\begin{equation}  \label{self-dual}
   \da_X \times \da_X  \lrto k \; : \;  \{f_{x,C}\}  \times \{ g_{x,C}  \} \longmapsto  \sum_{x \in C} \tr\nolimits_{k(x)/k} \circ
\res\nolimits_{x,C} (f_{x,C} \, g_{x,C}  \,  \omega)   \, \mbox{,}
\end{equation}
where $\{ f_{x,C} \}$ and $\{ g_{x,C}  \}$ are from ${\prod\limits_{x \in C}}' K_{x,C} = \da_X$, and if $K_{x, C} = \prod\limits_i K_i$, then
${ \res_{x, C} : \Omega^2_{K_{x,C}/k} \to k(x)}$ equals $\sum\limits_i \tr_{k_i/ k(x)}  \circ \res_{K_i} $, where
\begin{equation}  \label{arrows}
\res\nolimits_{K_i} \; :  \; \Omega^2_{K_i/k(x)}  \lrto \Omega^2_{K_i/k_i}  \lrto  \tilde{\Omega}^2_{K_i/k_i}   \lrto k_i  \, \mbox{,}
\end{equation}
where $K_i \simeq k_i((u))((t))$ and this isomorphism is a homeomorphism under natural topologies. The first map in formula~\eqref{arrows} is the natural map, the second map is the map to ``continuous differential forms'', i.e. to the quotient module by a $K_i$-submodule generated by elements ${f_1 df_2 \wedge df_3 - f_1 df_2 \wedge \left(\frac{\partial{f_3}}{\partial{u}} du +\frac{\partial{f_3}}{\partial{t}} dt  \right)}$, and the last map is
$$
\sum_{j,l} a_{j,l} u^j t^l du \wedge dt  \longmapsto a_{-1,-1}  \, \mbox{,}
$$
where $a_{j,l} \in k_i$. We note also that the sum~\eqref{self-dual} is finite.

We now give a theorem.

\begin{Th}  \label{Th2}
Let $\E$ be a locally free sheaf of $\oo_X$-modules of rank $n$ on a smooth projective surface $X$ over a perfect field $k$. We have
$$
f_{\E} =   -\frac{1}{2} K \cdot  c_1(\E)  - \frac{1}{2}c_1(\E)^2 + c_2(\E)
=
 -\frac{1}{2} K \cdot  c_1(\E) -{\rm ch}_2(\E)   \, \mbox{,}
$$
where $K \simeq \oo_X(\omega)$,   $\omega \in \Omega^2_{k(X)/k}$, $\omega \ne 0$.
\end{Th}
From Theorem~\ref{Th1} and Theorem~\ref{Th2} we immediately obtain a corollary.
\begin{cons}[Riemann-Roch theorem]  \label{Cor1}
We have
$$
\chi(\E)  - n \chi(\oo_X) = \frac{1}{2} c_1(\E) \cdot (c_1(\E) - K) - c_2(\E)  \, \mbox{.}
$$
\end{cons}
\begin{nt} \em
If $\E$ is an invertible sheaf on $X$, then  a number
$$p_a(\E)= 1 + \frac{1}{2} \E \cdot (\E + K)$$
is called {\em a  virtual arithmetic genus} of $\E$ (see, e.g.,~\cite[Ch.~IV, \S~2.8]{S} ).
Therefore, in this case  we have
$$
f_{\E} = 1 - p_a(\E)  \, \mbox{.}
$$
\end{nt}
\begin{proof}
It is enough to prove the first equality in the statement of Theorem~\ref{Th2}.
By the same arguments as in the beginning of the proof of Theorem~\ref{Th1} we obtain that  it is enough to suppose that $n =1$, which we will do.

For any $k$-vector subspace $V \subset \da_X$ we denote by $V^{\perp}$ the annihilator of $V$ in $\da_X$ with respect to the pairing~\eqref{self-dual}.
From the reciprocity laws on $X$ for residues of differential two-forms on two-dimensional local fields  (these reciprocity laws are ``along an irreducible curve'' and ``around a point''),  it is possible to obtain  (see~\cite{Par}) that for any divisor $D$ on $X$
\begin{equation}  \label{perp}
\da_{12}(D)^{\perp}= \da_{12}((\omega) - D) \, \mbox{,} \qquad \da_{02}^{\perp} = \da_{02}  \, \mbox{,}
\qquad \da_{01}^{\perp} = \da_{01}  \, \mbox{.}
\end{equation}

There is a canonical isomorphism of groups   (see~\cite[\S~5.5.5]{OsipPar1})
$$
\varphi \; : \;
\widetilde{GL_1(\da_X)}  \lrto \widetilde{GL_1(\da_X)}_{\da_{12}^{\perp}}  \, \mbox{,}
$$
where the central extension $\widetilde{GL_1(\da_X)}_{\da_{12}^{\perp}}$ is constructed similar to the central extension $\widetilde{GL_1(\da_X)}$,
but starting from the lattice $\da_{12}^{\perp}$ instead of the lattice $\da_{12}$ (cf.~Remark~\ref{Aut}). More precisely, $\widetilde{GL_1(\da_X)}_{\da_{12}^{\perp}}$ consists of all pairs $(g, d)$, where $g \in GL_1(\da_X)$ and
${d \in \Dim(\da_{12}^{\perp}  \mid g \da_{12}^{\perp})}$ with the multiplication law as in formula~\eqref{group law}. Explicitly,
isomorphism~$\varphi$
is given as
$$
\varphi((g,d))  = (g^{-1}, d)  \, \mbox{,}
$$
where we use canonical isomorphism $\Dim(\da_{12} \mid g \da_{12})  \simeq \Dim(\da_{12}^{\perp} \mid g^{-1} \da_{12}^{\perp})$, which is  based
on the equality $g^{-1} \da_{12}^{\perp} = (g \da_{12})^{\perp}  $  and on a canonical isomorphism \linebreak
$\Dim(E_1 \mid E_2) \simeq   \Dim (E_1^{\perp} \mid E_2^{\perp})$, where $E_1, E_2$ are lattices in $\da_X$. The last isomorphism    comes from the following chain of isomorphisms, where  $E_3 $ is a lattice such that $E_3 \subset E_1$, $E_3 \subset E_2$:
\begin{gather*}
\Dim(E_1 \mid E_2) \simeq \Dim(E_1 \mid E_3) \otimes_{\Z}  \Dim(E_3 \mid E_2) \simeq
\Dim(E_1 / E_3)^* \otimes_{\Z} \Dim(E_2/ E_3) \simeq  \\ \simeq
\Dim(E_3^{\perp} / E_1^{\perp})  \otimes_{\Z} \Dim(E_3^{\perp}/ E_2^{\perp})^* \simeq \Dim(E_1^{\perp} \mid E_2^{\perp})  \, \mbox{,}
\end{gather*}
where ${}^*$ means the dual $\Z$-torsor and we used  canonical isomorphisms~\eqref{dim-iso} and~\eqref{dual}.

We note that isomorphism $\varphi$ restricted to the central subgroup $\Z \subset \widetilde{GL_1(\da_X)} $ is the identity morphism to the central subgroup $\Z \subset \widetilde{GL_1(\da_X)} $.

The chosen basis $e_0$ of $\E$ at the generic point of $X$ gives the embedding of $\E$ to the constant sheaf $k(X)$ on $X$. Therefore $\E = \oo_X(D)$
for some divisor $D$ on $X$. We fix also the other bases for $\E$ and hence the transitions elements $\alpha_{ij}$ for $\E$, where $i \ne j $ from
$\{ 1,2 \}$, as in Section~\ref{trivializations}.

Therefore, we have
$$
f_{\E} = \varphi(f_{\E}) =
\varphi(\widetilde{\alpha_{02}}  \cdot \widetilde{\alpha_{21}} \cdot  \widetilde{\alpha_{10}})=
\varphi(\widetilde{\alpha_{02}})  \cdot \varphi(\widetilde{\alpha_{21}}) \cdot  \varphi(\widetilde{\alpha_{10}})  \, \mbox{.}
$$
Hence,  the proof of the theorem will follow from the proof of the following formula
$$
\varphi(\widetilde{\alpha_{01}})  \cdot \varphi(\widetilde{\alpha_{12}}) \cdot  \varphi(\widetilde{\alpha_{20}}) =
\widetilde{\alpha_{02}}  \cdot \widetilde{\alpha_{21}} \cdot  \widetilde{\alpha_{10}}
+ (\E, \E) + (\E, \oo_X(\omega))  \, \mbox{,}
$$
where we used that
$$
-(\widetilde{\alpha_{02}}  \cdot \widetilde{\alpha_{21}} \cdot  \widetilde{\alpha_{10}}) =
\widetilde{\alpha_{01}}  \cdot \widetilde{\alpha_{12}} \cdot  \widetilde{\alpha_{20}}  \, \mbox{.}
$$

From formula~\eqref{inter-index-new} for the intersection index of invertible sheaves we have
$$
(\E, \E) = \langle \widetilde{\alpha_{01}}, \widetilde{\alpha_{02}}   \rangle =
\widetilde{\alpha_{01}} \cdot  \widetilde{\alpha_{02}} \cdot \widetilde{\alpha_{10}} \cdot  \widetilde{\alpha_{20}}  \, \mbox{.}
$$
Therefore we obtain
\begin{gather*}
\widetilde{\alpha_{02}}  \cdot \widetilde{\alpha_{21}} \cdot  \widetilde{\alpha_{10}} + (\E, \E)=
\widetilde{\alpha_{02}}  \cdot \widetilde{\alpha_{21}} \cdot  \widetilde{\alpha_{10}} \cdot
\widetilde{\alpha_{01}} \cdot  \widetilde{\alpha_{02}} \cdot \widetilde{\alpha_{10}} \cdot  \widetilde{\alpha_{20}}
= \\ =
\widetilde{\alpha_{02}}  \cdot \left( \widetilde{\alpha_{21}} \cdot  \widetilde{\alpha_{02}} \cdot
\widetilde{\alpha_{10}}   \right) \cdot  \widetilde{\alpha_{02}}^{-1}=
\widetilde{\alpha_{21}} \cdot  \widetilde{\alpha_{02}}   \cdot \widetilde{\alpha_{10}}= \\ =
\widetilde{\alpha_{12}}  \cdot \left(  \widetilde{\alpha_{21}} \cdot  \widetilde{\alpha_{02}} \cdot
\widetilde{\alpha_{10}}     \right) \widetilde{\alpha_{12}}^{-1} =
\widetilde{\alpha_{02}}  \cdot \widetilde{\alpha_{10}} \cdot  \widetilde{\alpha_{21}}  \, \mbox{,}
\end{gather*}
where we used that the conjugation does not change the result. Hence, to prove the theorem it is enough to prove the formula
$$
\varphi(\widetilde{\alpha_{01}})  \cdot \varphi(\widetilde{\alpha_{12}}) \cdot  \varphi(\widetilde{\alpha_{20}})  -
\widetilde{\alpha_{02}}  \cdot \widetilde{\alpha_{10}} \cdot  \widetilde{\alpha_{21}}  = (\E, \oo_X(\omega)) \, \mbox{.}
$$

Let $c = \widetilde{\alpha_{02}}  \cdot \widetilde{\alpha_{10}} \cdot  \widetilde{\alpha_{21}}$. Hence
$
c \cdot \widetilde{\alpha_{12}} = \widetilde{\alpha_{02}}  \cdot \widetilde{\alpha_{10}}
$. Therefore, by Lemma~\ref{lem1} we have
$$
c = \nu_{\da_{12}, \da_{12}(\E)}  - \mu_{\da_{12}, \da_{12}(\E)}  \, \mbox{.}
$$

Now we calculate $d =  \varphi(\widetilde{\alpha_{01}})  \cdot \varphi(\widetilde{\alpha_{12}}) \cdot  \varphi(\widetilde{\alpha_{20}}) =
 \varphi(\widetilde{\alpha_{12}}) \cdot  \varphi(\widetilde{\alpha_{20}})  \cdot \varphi(\widetilde{\alpha_{01}}) $.
We have
\begin{equation}   \label{form-d}
 d \cdot \varphi(\widetilde{\alpha_{21}}) = \varphi(\widetilde{\alpha_{20}})  \cdot \varphi(\widetilde{\alpha_{01}})
\end{equation}
From the construction of $\varphi$ and formulas~\eqref{perp} we have
\begin{gather*}
\varphi (\widetilde{\alpha_{01}})= \varphi ((\alpha_{01}, \nu_{\da_{12}, \alpha_{01} \da_{12}})) =
(\alpha_{10}, \nu_{\da_{12}^{\perp}, \alpha_{10} \da_{12}^{\perp}} ) =
(\alpha_{10}, \nu_{\da_{12}((\omega)), \da_{12}((\omega) + D)} )  \\
\varphi (\widetilde{\alpha_{02}})= \varphi ((\alpha_{02}, \mu_{\da_{12}, \alpha_{02} \da_{12}})) =
(\alpha_{20}, \mu_{\da_{12}^{\perp}, \alpha_{20} \da_{12}^{\perp}} ) =
(\alpha_{20}, \mu_{\da_{12}((\omega)), \da_{12}((\omega) + D)} )  \\
\varphi(\widetilde{\alpha_{21}})= \varphi((\alpha_{21}, 0  ))= (\alpha_{12}, 0)
 \, \mbox{.}
\end{gather*}
From these formulas and formula~\eqref{form-d}, by the same reason as in calculation of $c$ above (see  the proof of Lemma~\ref{lem1})
we obtain that
$$
d = \nu_{\da_{12}((\omega)), \da_{12}((\omega) +D)}  - \mu_{\da_{12}((\omega)) , \da_{12}((\omega) + D)} \, \mbox{.}
$$
Now we have
\begin{gather*}
d -c = \left(\nu_{\da_{12}((\omega)), \da_{12}((\omega) +D)}  - \mu_{\da_{12}((\omega)) , \da_{12}((\omega) + D)} \right) -
\left( \nu_{\da_{12}, \da_{12}(D)}  - \mu_{\da_{12}, \da_{12}(D)} \right) = \\
=\left(\nu_{\da_{12}((\omega)), \da_{12}((\omega) +D)}  - \mu_{\da_{12}((\omega)) , \da_{12}((\omega) + D)} \right) +
\left(\mu_{\da_{12}, \da_{12}(D)}  - \nu_{\da_{12}, \da_{12}(D)}  \right) + \\ +
\left(
\nu_{\da_{12}(D), \da_{12}(\omega)}  - \nu_{\da_{12}(D), \da_{12}(\omega)}
\right) =
\left(\mu_{\da_{12}, \da_{12}(D)}  - \nu_{\da_{12}, \da_{12}(D)}  \right) + \\ +
\left(
\nu_{\da_{12}(D), \da_{12}(\omega)}  - \nu_{\da_{12}(D), \da_{12}(\omega)}
\right) +
\left(\nu_{\da_{12}((\omega)), \da_{12}((\omega) +D)}  - \mu_{\da_{12}((\omega)) , \da_{12}((\omega) + D)} \right) = \\
= \mu_{\da_{12}, \da_{12}(D)}  \otimes \nu_{\da_{12}(D), \da_{12}(\omega)} \otimes \nu_{\da_{12}((\omega)), \da_{12}((\omega) +D)} - \\
-
\nu_{\da_{12}, \da_{12}(D)} \otimes \nu_{\da_{12}(D), \da_{12}(\omega)}  \otimes \mu_{\da_{12}((\omega)) , \da_{12}((\omega) + D)} = \\ =
\mu_{\da_{12}, \da_{12}(D)} \otimes \nu_{\da_{12}(D),  \da_{12}((\omega) +D)}  -
\nu_{\da_{12},  \da_{12}(\omega)}  \otimes \mu_{\da_{12}((\omega)) , \da_{12}((\omega) + D)} \, \mbox{.}
\end{gather*}
Hence and by Proposition~\ref{prop-inter-index}, we obtain $d -c = (D, (\omega))$.

\end{proof}

\vspace{0.3cm}

\noindent Steklov Mathematical Institute of Russsian Academy of Sciences,  Moscow, Russian Federation, {\em and}

\noindent National Research University Higher School of Economics,  Moscow, Russian Federation,
{\em and}

\noindent National University of Science and Technology ``MISiS'',   Moscow, Russian Federation

\vspace{0.3cm}

\noindent {\it E-mail:}  ${d}_{-} osipov@mi{-}ras.ru$

\end{document}